\documentclass[11pt]{article}
\usepackage[margin=1in]{geometry}
\usepackage{amsmath,amssymb,amsthm,mathtools,mathrsfs}
\usepackage[T1]{fontenc}
\usepackage{lmodern}
\usepackage{microtype}
\usepackage{hyperref}
\usepackage{enumitem}

\newtheorem{theorem}{Theorem}[section]
\newtheorem{proposition}[theorem]{Proposition}
\newtheorem{lemma}[theorem]{Lemma}
\newtheorem{corollary}[theorem]{Corollary}
\theoremstyle{definition}
\newtheorem{definition}[theorem]{Definition}
\theoremstyle{remark}
\newtheorem{remark}[theorem]{Remark}

\newcommand{\Q}{\mathbb Q}
\newcommand{\R}{\mathbb R}
\newcommand{\C}{\mathbb C}

\newcommand{\Z}{\mathbb Z}
\newcommand{\A}{\mathbb A}
\newcommand{\Id}{\mathrm{Id}}

\newcommand{\Hilb}[2]{(#1,#2)}
\newcommand{\leg}[2]{\left(\frac{#1}{#2}\right)}

\newcommand{\Sch}{\mathcal S}
\newcommand{\sgn}{\operatorname{sgn}}
\newcommand{\diag}{\operatorname{diag}}
\newcommand{\Span}{\operatorname{Span}}

\title{Reciprocity and the Maslov Phase}
\author{Jonathan Holland}
\date{\today}

\begin{document}
\maketitle

\begin{abstract}
We give a metaplectic proof of Hilbert reciprocity, and hence of quadratic reciprocity, in which the local phase is the Kashiwara--Maslov phase of a triple of Lagrangians \cite{LionVergne,ThomasMaslov}.  In rank two the phase of the ordered triple $(L_\infty,L_a,L_0)$ is the one-dimensional Weil index $\gamma_v(a)$.  The local Hilbert symbol appears as the defect of strict multiplicativity of these phases:
\[
\Hilb{a}{b}_v
=
\frac{\gamma_v(a)\gamma_v(b)}{\gamma_v(1)\gamma_v(ab)}.
\]
The global step compares the local and adelic realizations of a single Bruhat word for the diagonal torus elements $m(a)\in \mathrm{SL}_2(\Q)$.  Locally the raw Bruhat-word lift carries the normalization factor determined by the chosen quadratic convention.  These operators form a projective representation of the diagonal torus with defect
\[
\mu_v(a,b)
=
\frac{\gamma_v(a)\gamma_v(b)}{\gamma_v(1)\gamma_v(ab)}.
\]
For rational adelic data, the normalized Bruhat word is multiplicative.  The reciprocity law states that the total defect $\prod_v\mu_v(a,b)$ is $1$.  Combined with the local bridge above, this yields Hilbert reciprocity, while quadratic reciprocity is then the specialization to the pair of odd primes $(p,q)$.
\end{abstract}

\section{Introduction}

Quadratic reciprocity is usually presented either as a theorem about quadratic residues or as an identity between Gauss sums. While elementary, these formulations tend to make the reciprocity sign look like the output of a clever calculation. This paper gives a metaplectic interpretation of the reciprocity laws.  The local quantity is the Kashiwara--Maslov phase attached to a triple of Lagrangians in a symplectic plane \cite{LionVergne,ThomasMaslov}.  In rank two, for the ordered triple
\[
(L_\infty,L_a,L_0),
\]
that phase is the one-dimensional Weil index $\gamma_v(a)$.  The local Hilbert symbol then appears as the failure of these phases to multiply strictly:
\[
\Hilb{a}{b}_v
=
\frac{\gamma_v(a)\gamma_v(b)}{\gamma_v(1)\gamma_v(ab)}.
\]
Reciprocity states that the \emph{total Maslov defect} of a rational loop is trivial \cite{LionVergne}.

The classical sign is the contribution at the place $2$ to the product of local metaplectic phases.  The local normalized Bruhat-word lifts define a projective system, but the global product has trivial defect.

The proof proceeds in four steps.
\begin{enumerate}[label=\arabic*.]
\item In a symplectic plane, the Kashiwara quadratic form of a triple of pairwise transverse Lagrangians is one-dimensional, and for the ordered triple $(L_\infty,L_a,L_0)$ it is the form $\langle a\rangle$ \cite{ThomasMaslov,PerrinMaslov}.
\item The normalized triple intertwiner in the Schr\"odinger models of the Heisenberg representation is the Weil index of that Kashiwara form \cite{Weil}.  Thus the local Kashiwara--Maslov phase of $(L_\infty,L_a,L_0)$ is $\gamma_v(a)$.
\item For each place $v$ of $\Q$, the local multiplicativity defect of the Weil index is the Hilbert symbol:
\[
\gamma_v(a)\gamma_v(b)=\gamma_v(1)\gamma_v(ab)\Hilb{a}{b}_v.
\]
The local Hilbert symbol is the multiplicativity defect of the local Maslov phases.
\item The global cancellation law is proved at the level of the defect cocycle.  For every place $v$, the normalized local Bruhat-word operators attached to the diagonal torus form a projective representation whose defect is
\[
\mu_v(a,b)=\frac{\gamma_v(a)\gamma_v(b)}{\gamma_v(1)\gamma_v(ab)}.
\]
On rational adelic data, the normalized Bruhat word is multiplicative.  Comparing the local and global multiplication laws shows that $\prod_v\mu_v(a,b)=1$.
\end{enumerate}

The proof identifies reciprocity with cancellation of the local Maslov defects.  The local normalized Bruhat-word lifts form a projective system, and rational adelic data have trivial total defect.

The rank-two identification of the Maslov cocycle with the one-dimensional Weil index is a standard metaplectic input. The global product formula is proved in the adelic Schr\"odinger model from the defect cocycle of the local Bruhat-word operators. The global inputs are the triviality of the standard additive character on $\Q$, the product formula $|a|_\A=1$, and Poisson summation for $\Q\subset\A$.

The final specialization to distinct odd primes $p$ and $q$ recovers quadratic reciprocity by evaluating the four places that can contribute: $p$, $q$, $2$, and $\infty$.  The two odd primes contribute the Legendre symbols, the real place is trivial, and the classical sign
\[
(-1)^{\frac{(p-1)(q-1)}4}
\]
is the $2$-adic component of the global cancellation law. The paper thus derives the reciprocity laws and identifies the geometric source of their signs.

\section{Classical warm-up on the real line}\label{sec:warmup}

Before turning to the local and adelic Maslov formalism, we examine the corresponding classical calculation in the ordinary Schr\"odinger model on $L^2(\R)$.  In this warm-up, quadratic reciprocity appears as a statement about metaplectic transport between rational Lagrangians, coupled to the lattice state
\[
\Theta:=\sum_{n\in\Z}\delta_n.
\]
Gauss sums appear as matrix coefficients of real metaplectic operators on finite residue sectors of rational lattice states.

Let
\[
\omega\big((x,\xi),(x',\xi')\big)=x\xi'-x'\xi
\]
be the standard symplectic form on $\Q^2\subset \R^2$, and for $t\in\Q\cup\{\infty\}$ write
\[
L_t:=\Q(1,t),\qquad L_\infty:=\Q(0,1).
\]
Thus a rational slope $t=a/c$ determines a rational Lagrangian line $L_{a/c}$.  The real metaplectic representation is generated by the standard lifts of
\[
S=\begin{pmatrix}0&1\\-1&0\end{pmatrix},\qquad
T_t=\begin{pmatrix}1&t\\0&1\end{pmatrix},\qquad
D_a=\begin{pmatrix}a&0\\0&a^{-1}\end{pmatrix},
\]
acting on Schwartz functions and tempered distributions by
\[
(\mathcal Ff)(x)=\int_\R e^{2\pi ixy}f(y)\,dy,
\qquad
(\mathcal M_t f)(x)=e^{\pi i t x^2}f(x),
\qquad
(\mathcal D_a f)(x)=|a|^{1/2}f(ax).
\]
Here $\mathcal M_t$ is the quadratic Schr\"odinger evolution, while Poisson summation gives the fixed-vector relation
\[
\mathcal F\Theta=\Theta.
\]
The theta-comb couples the continuous metaplectic representation to the arithmetic lattice.

The real model also contains a projective sign attached to triples of rational lines.  If $v_t=(1,t)$ and $v_\infty=(0,1)$, then for three pairwise transverse rational lines one may form the sign
\[
\kappa(L_{t_1},L_{t_2},L_{t_3})
=
\sgn\!\Big(
\omega(v_{t_1},v_{t_2})\,
\omega(v_{t_2},v_{t_3})\,
\omega(v_{t_3},v_{t_1})
\Big).
\]
In dimension one this is the sign part of the Kashiwara index: it records the cyclic ordering of three real Lagrangians.  This sign does not encode arithmetic data such as $\leg{a}{c}$.  The arithmetic enters through the action of metaplectic transport toward a rational line on the lattice state $\Theta$.

For an odd integer $c>0$, the denominator-$c$ finite sector is generated by the residue-class combs.  For $r\in \Z/c\Z$, set
\[
E_r^{(c)}:=
\sum_{k\in\Z}
\delta_{r+kc}.
\]
Then
\[
W_c:=\Span\{E_r^{(c)}:r\in \Z/c\Z\}
\]
is preserved by the quadratic phase $\mathcal M_{2a/c}$, and one has
\[
\mathcal M_{2a/c}E_r^{(c)}
=
\exp\!\left(2\pi i\frac{a r^2}{c}\right)E_r^{(c)}.
\]
The factor $2a/c$ is forced by our real metaplectic convention
\[
(\mathcal M_t f)(x)=e^{\pi i t x^2}f(x).
\]
It is also what makes the phase descend to $\Z/c\Z$; for odd $c$, the expression $e^{\pi i a r^2/c}$ would not be well-defined modulo $c$.

Fourier transform carries the residue combs to the dual denominator-$c$ combs:
\[
\mathcal F E_r^{(c)}
=
\frac1c\sum_{m\in\Z}e^{2\pi i rm/c}\,\delta_{m/c}.
\]
After the natural rescaling of $c^{-1}\Z/\Z$ back to $\Z/c\Z$, the induced finite operator is the usual finite Fourier matrix
\[
(\mathcal F_c)_{m,r}=c^{-1/2}e^{2\pi i rm/c}.
\]
Thus the real metaplectic word $\mathcal F\mathcal M_{2a/c}$ has, on the denominator-$c$ finite quotient, the operator
\[
\mathcal F_c\circ \diag\!\left(e^{2\pi i a r^2/c}\right)_{r\bmod c}.
\]
The mod-$c$ quadratic exponential sum is the distinguished coefficient of this finite operator on the finite residue sector determined by the lattice.

For this warm-up we use a fixed word, rather than an arbitrary lift.  Define
\[
\Gamma(L_\infty,L_{2a/c})
\]
to be the unnormalized coefficient of the unshifted residue class in the finite operator
\[
\mathcal F_c\circ \diag\!\left(e^{2\pi i a r^2/c}\right)_{r\bmod c},
\]
after the natural rescaling of $c^{-1}\Z/\Z$ back to $\Z/c\Z$.  Thus the finite word, the residue-sector normalization, and the lift are fixed by definition.  Unwinding the coefficient gives the ordinary quadratic Gauss sum
\[
\mathcal G(a,c):=\sum_{x\, (\mathrm{mod}\, c)} e^{2\pi i a x^2/c}.
\]
Equivalently, the unitary finite matrix coefficient is $c^{-1/2}\mathcal G(a,c)$, while the lattice-comb normalization used for $\Gamma$ is the unnormalized coefficient
\[
\Gamma(L_\infty,L_{2a/c})=\mathcal G(a,c).
\]
The notation $\Gamma(L_\infty,L_{2a/c})$ records the geometric origin of this fixed transport coefficient.

The same coefficient can be described in the residue basis.  Transport the comb $\Theta=\sum_{n\in\Z}\delta_n$ from the $L_\infty$-model to the $L_{2a/c}$-model by the normalized metaplectic operator, and project to the denominator-$c$ residue sector.  In the basis $\{E_r^{(c)}\}_{r\bmod c}$, followed by the natural rescaling of $c^{-1}\Z/\Z$ back to $\Z/c\Z$, the coefficient of the unshifted residue class is $\mathcal G(a,c)$.

We use two classical identities.

\smallskip
\noindent\textbf{Variation in the numerator.}
If $c$ is odd and $(a,c)=1$, then
\[
\Gamma(L_\infty,L_{2a/c})
=
\sum_{x\, (\mathrm{mod}\, c)} e^{2\pi i a x^2/c}
=:\mathcal G(a,c).
\]
The classical identity
\[
\mathcal G(a,c)=\leg{a}{c}\mathcal G(1,c)
\]
expresses the dependence on the numerator.  If $a\equiv u^2\pmod c$, the change of variables $x\mapsto ux$ gives
\[
\mathcal G(a,c)=\mathcal G(1,c).
\]
For general $(a,c)=1$, the standard transformation law for quadratic Gauss sums gives the Jacobi-symbol factor.  Geometrically, for fixed odd denominator $c$, moving the target line $L_{2/c}$ through the rational family $L_{2a/c}$ multiplies the lattice response by this factor:
\[
\Gamma(L_\infty,L_{2a/c})=\leg{a}{c}\,\Gamma(L_\infty,L_{2/c}).
\]

\smallskip
\noindent\textbf{Evaluation at $a=1$.}
The theta transformation law
\[
\vartheta(-1/z)=(-iz)^{1/2}\vartheta(z)
\]
shows that for odd $c>0$,
\[
\Gamma(L_\infty,L_{2/c})=\mathcal G(1,c)=\varepsilon_c\sqrt c,
\qquad
\varepsilon_c=
\begin{cases}
1,&c\equiv 1\pmod 4,\\
 i,&c\equiv 3\pmod 4.
\end{cases}
\]
Here $\sqrt c$ is the cusp-scaling factor and $\varepsilon_c$ is the metaplectic square-root phase for transport from $L_\infty$ to the line $L_{2/c}$. Equivalently, comparing the geometric and Fourier-theoretic asymptotics of the theta function at the corresponding odd cusp recovers the standard evaluation
\[
\sum_{r\, (\mathrm{mod}\, c)} e^{2\pi i r^2/c}=\varepsilon_c\sqrt c.
\]
The unitary finite coefficient is $c^{-1/2}\mathcal G(1,c)$, while the lattice-comb coefficient $\Gamma(L_\infty,L_{2/c})$ is the unnormalized Gauss sum.

Combining the two laws gives, for odd $c$ and $(a,c)=1$,
\[
\Gamma(L_\infty,L_{2a/c})=\leg{a}{c}\,\varepsilon_c\sqrt c.
\]
The formula separates the rational line, the metaplectic square-root phase, and the quadratic character.

The finite sum is the coordinate expression of a metaplectic transport coefficient for rational Lagrangians tested against the integer lattice.  The calculation identifies this coefficient with the classical Gauss sum; the adelic argument below explains the product law for the associated local defects.

For distinct odd primes $p$ and $q$, the Chinese remainder theorem identifies the denominator-$pq$ residue sector $W_{pq}$ with $W_p\otimes W_q$.  Under this decomposition, and after the square changes of variables that replace $q^{-1}$ by $q$ modulo $p$ and $p^{-1}$ by $p$ modulo $q$, the Gauss sum factors as
\[
\mathcal G(1,pq)=\mathcal G(q,p)\mathcal G(p,q).
\]
Equivalently,
\[
\Gamma(L_\infty,L_{2/pq})
=
\Gamma(L_\infty,L_{2q/p})\,\Gamma(L_\infty,L_{2p/q}).
\]
Substituting the preceding formula gives
\[
\varepsilon_{pq}\sqrt{pq}
=
\leg{q}{p}\,\varepsilon_p\sqrt p\;\leg{p}{q}\,\varepsilon_q\sqrt q,
\]
and hence
\[
\leg{p}{q}\leg{q}{p}
=
\frac{\varepsilon_{pq}}{\varepsilon_p\varepsilon_q}
=
(-1)^{\frac{(p-1)(q-1)}{4}}.
\]
In the real model, the reciprocity sign is the quotient of the metaplectic square-root phases attached to $L_{2/p}$, $L_{2/q}$, and $L_{2/pq}$.

The warm-up identifies the classical Gauss sums as metaplectic transport coefficients.  The following sections replace the finite calculation by the local Maslov formalism and the adelic cancellation law.

\section{Kashiwara forms for triples of slopes}

Let $K$ be a field of characteristic different from $2$, let
\[
V=K^2,
\qquad
\omega\big((x,y),(x',y')\big)=xy'-yx',
\]
and for $a\in K$ set
\[
e_a:=(1,a),\qquad L_a:=K e_a.
\]
We also write
\[
e_\infty:=(0,1),\qquad L_\infty:=K e_\infty.
\]
Then
\[
\omega(e_a,e_b)=b-a,
\qquad
\omega(e_\infty,e_a)=-1.
\]

\begin{definition}
For a triple of Lagrangians $L_1,L_2,L_3\subset V$, define
\[
K(L_1,L_2,L_3):=\{(x_1,x_2,x_3)\in L_1\oplus L_2\oplus L_3:x_1+x_2+x_3=0\}.
\]
On this subspace define the \emph{Kashiwara quadratic form}
\[
q_{L_1,L_2,L_3}(x_1,x_2,x_3):=\omega(x_1,x_2).
\]
Since $x_1+x_2+x_3=0$, one has
\[
\omega(x_1,x_2)=\omega(x_2,x_3)=\omega(x_3,x_1),
\]
so the definition is independent of the chosen pair.
\end{definition}

\begin{proposition}\label{prop:finite-slopes}
Let $a,b,c\in K$ be distinct.  Then $K(L_a,L_b,L_c)$ is one-dimensional, spanned by
\[
\xi_{a,b,c}:=\big((b-c)e_a,(c-a)e_b,(a-b)e_c\big),
\]
and
\[
q_{L_a,L_b,L_c}(\xi_{a,b,c})=-(a-b)(b-c)(c-a).
\]
Hence
\[
q_{L_a,L_b,L_c}\cong\left\langle -(a-b)(b-c)(c-a)\right\rangle.
\]
\end{proposition}

\begin{proof}
Write
\[
x_1=t_1 e_a,\qquad x_2=t_2 e_b,\qquad x_3=t_3 e_c.
\]
The relation $x_1+x_2+x_3=0$ becomes
\[
t_1+t_2+t_3=0,
\qquad
at_1+bt_2+ct_3=0.
\]
Eliminating $t_3$ gives
\[
(a-c)t_1+(b-c)t_2=0.
\]
A nonzero solution is
\[
t_1=b-c,\qquad t_2=c-a,\qquad t_3=a-b,
\]
so $\xi_{a,b,c}$ spans $K(L_a,L_b,L_c)$.

Now
\begin{align*}
q_{L_a,L_b,L_c}(\xi_{a,b,c})
&=\omega\big((b-c)e_a,(c-a)e_b\big)\\
&=(b-c)(c-a)\omega(e_a,e_b)\\
&=(b-c)(c-a)(b-a)\\
&=-(a-b)(b-c)(c-a).
\end{align*}
This gives the stated coefficient.
\end{proof}

\begin{proposition}\label{prop:infty-slope}
Let $a,b\in K$ be distinct.  Then $K(L_\infty,L_a,L_b)$ is one-dimensional, spanned by
\[
\xi_{\infty,a,b}:=\big((b-a)e_\infty,e_a,-e_b\big),
\]
and
\[
q_{L_\infty,L_a,L_b}(\xi_{\infty,a,b})=a-b.
\]
Hence
\[
q_{L_\infty,L_a,L_b}\cong\langle a-b\rangle.
\]
In particular,
\[
q_{L_\infty,L_\alpha,L_0}\cong\langle \alpha\rangle.
\]
\end{proposition}

\begin{proof}
Write
\[
x_1=s e_\infty,\qquad x_2=t e_a,\qquad x_3=u e_b.
\]
The equation $x_1+x_2+x_3=0$ gives
\[
t+u=0,
\qquad
s+at+bu=0.
\]
Hence $u=-t$ and $s=(b-a)t$.  Taking $t=1$ yields the spanning vector $\xi_{\infty,a,b}$.

Then
\begin{align*}
q_{L_\infty,L_a,L_b}(\xi_{\infty,a,b})
&=\omega\big((b-a)e_\infty,e_a\big)\\
&=(b-a)\omega(e_\infty,e_a)\\
&=-(b-a)=a-b.
\end{align*}
For $(\infty,\alpha,0)$ this specializes to $\langle\alpha\rangle$.  For comparison, the reversed finite-slope order gives $q_{L_\infty,L_0,L_\alpha}\cong\langle-\alpha\rangle$.
\end{proof}

\section{The local Maslov cocycle and the Weil index}

Let $F$ be a local field of characteristic different from $2$, let $\psi:F\to\C^\times$ be a nontrivial additive character, and let $V$ be a two-dimensional symplectic $F$-vector space.

For each Lagrangian $L\subset V$, let $\mathcal H_L$ denote the corresponding Schr\"odinger model of the Heisenberg representation with central character $\psi$.  If $L$ and $M$ are transverse, there is a canonically normalized intertwiner
\[
T_{M,L}:\mathcal H_L\longrightarrow \mathcal H_M.
\]

\begin{theorem}[Maslov cocycle = Weil index \cite{LionVergne,Weil}]\label{thm:maslov-weil}
Let $L_1,L_2,L_3$ be pairwise transverse Lagrangians in a two-dimensional symplectic space over $F$.  Then
\[
T_{L_3,L_2}T_{L_2,L_1}T_{L_1,L_3}
=
\gamma_\psi\big(q_{L_1,L_2,L_3}\big)\,\Id,
\]
where $q_{L_1,L_2,L_3}$ is the Kashiwara quadratic form of the triple and $\gamma_\psi(q)$ is the Weil index of the quadratic form $q$.
\end{theorem}

\begin{proof}
  This is the rank-two case of the metaplectic cocycle formula; see Rao~\cite[Theorem~5.1]{Rao} or Lion--Vergne~\cite[Chapter~II, \S2]{LionVergne}. The triple composition depends on the Kashiwara quadratic space $K(L_1,L_2,L_3)$. For a one-dimensional form $\langle a\rangle$, the scalar is $\gamma_\psi(a)$.
\end{proof}

\begin{remark}
  Theorem~\ref{thm:maslov-weil} is the metaplectic input. The product formula is proved in the adelic Schr\"odinger model, and the arithmetic input is the local multiplicativity law together with the explicit Hilbert-symbol evaluations.
\end{remark}

For $a\in F^\times$ we write
\[
\gamma_\psi(a):=\gamma_\psi(\langle a\rangle)=\gamma_\psi(a x^2).
\]
Then Proposition~\ref{prop:infty-slope} and Theorem~\ref{thm:maslov-weil} identify the local phase.

\begin{corollary}\label{cor:basic-phase}
For every $a\in F^\times$, the triple phase of the ordered triple $(L_\infty,L_a,L_0)$ is the Weil index $\gamma_\psi(a)$:
\[
T_{L_0,L_a}T_{L_a,L_\infty}T_{L_\infty,L_0}=\gamma_\psi(a)\,\Id.
\]
Thus the one-dimensional Weil index is a Kashiwara--Maslov phase \cite{ThomasMaslov}.
\end{corollary}

\begin{proof}
By Proposition \ref{prop:infty-slope}, the Kashiwara form of $(L_\infty,L_a,L_0)$ is $\langle a\rangle$.  The order matters: with the present convention $(L_\infty,L_0,L_a)$ gives $\langle -a\rangle$.  The result is therefore immediate from Theorem \ref{thm:maslov-weil}.
\end{proof}

\section{Local Weil indices: conventions and multiplicativity}\label{sec:local-weil}

From now on $v$ denotes a place of $\Q$, and $F=\Q_v$.

\subsection{Standard additive characters and self-dual measures}

We use the standard local additive characters:
\[
\psi_\infty(x)=e^{2\pi i x}
\qquad (x\in\R),
\]
and for a prime $p$,
\[
\psi_p(x)=e^{-2\pi i\{x\}_p}
\qquad (x\in\Q_p),
\]
where $\{x\}_p\in\Q/\Z$ denotes the $p$-adic fractional part.  The global character on the adeles is
\[
\psi(x)=\prod_v\psi_v(x_v),\qquad x=(x_v)\in\A,
\]
and it is trivial on the diagonal copy of $\Q$.

For each $v$ we choose the Haar measure $dx_v$ on $\Q_v$ self-dual with respect to $\psi_v$.  Explicitly, $dx_\infty$ is Lebesgue measure on $\R$, and for finite $p$ the self-dual measure is normalized by
\[
\operatorname{vol}(\Z_p,dx_p)=1.
\]
The adelic measure is the restricted product
\[
dx=\prod_v dx_v,
\]
which is self-dual with respect to the global character $\psi$.

\subsection{Definition of the one-dimensional Weil index}

Let $a\in\Q_v^\times$.  The one-dimensional Weil index $\gamma_v(a)$ is the unique scalar for which the Fourier transform of the quadratic exponential has the form
\begin{equation}\label{eq:local-gaussian}
\int_{\Q_v}\widehat{\phi}(x)\,\psi_v(ax^2)\,dx
=
\gamma_v(a)|2a|_v^{-1/2}
\int_{\Q_v}\phi(x)\,\psi_v\!\left(-\frac{x^2}{4a}\right)dx
\end{equation}
for every Schwartz--Bruhat function $\phi\in\Sch(\Q_v)$.  Here
\[
\widehat{\phi}(y):=\int_{\Q_v}\phi(x)\psi_v(xy)\,dx
\]
is the self-dual Fourier transform.

The existence and uniqueness of $\gamma_v(a)$ are standard.  The normalization above is the one compatible with the metaplectic cocycle and with the standard adelic Fourier transform.  In particular, $\gamma_v(a)$ depends only on the square-class of $a$ and satisfies $|\gamma_v(a)|=1$.

\begin{lemma}\label{lem:square-invariance}
For $a,c\in\Q_v^\times$,
\[
\gamma_v(ac^2)=\gamma_v(a).
\]
Hence $\gamma_v$ descends to a function on $\Q_v^\times/\Q_v^{\times 2}$.
\end{lemma}

\begin{proof}
Replace $x$ by $cx$ in \eqref{eq:local-gaussian}.  Since the self-dual measure transforms by $d(cx)=|c|_v\,dx$ and the prefactor on the right changes by $|c^2|_v^{-1/2}=|c|_v^{-1}$, the resulting identity is the defining identity for $\gamma_v(a)$.  Uniqueness therefore implies $\gamma_v(ac^2)=\gamma_v(a)$.
\end{proof}

\subsection{The local multiplicativity law}

We next prove the identity relating the defect of multiplicativity of $\gamma_v$ to the local Hilbert symbol.  Recall that for $a,b\in\Q_v^\times$, the Hilbert symbol $\Hilb{a}{b}_v\in\{\pm1\}$ is characterized by
\[
\Hilb{a}{b}_v=1
\iff
z^2=ax^2+by^2\text{ has a nonzero solution in }\Q_v^3.
\]
Equivalently, $\Hilb{a}{b}_v=1$ iff $b$ is a norm from $\Q_v(\sqrt a)$ \cite[Ch.~III, \S1]{SerreLocal}.

\begin{theorem}\label{thm:local-mult}
For every place $v$ of $\Q$ and all $a,b\in\Q_v^\times$ one has
\begin{equation}\label{eq:local-multiplicativity}
\gamma_v(a)\gamma_v(b)=\gamma_v(1)\gamma_v(ab)\Hilb{a}{b}_v.
\end{equation}
\end{theorem}

\begin{proof}
This is proved in Appendix~\ref{app:rank-two-local-bridge}, Theorem~\ref{thm:app-local-bridge}.  The proof there reduces the binary identity to the explicit Hasse-invariant formula for Weil indices of diagonal quadratic forms and records the normalization used here.
\end{proof}

\begin{remark}
  We use \eqref{eq:local-multiplicativity} as the local input.  It is the standard binary specialization of the Hasse-invariant formula for Weil indices.
\end{remark}

\section{Global splitting and the product formula for the local defects}\label{sec:global-product}

We now prove the global identity needed for reciprocity, by showing that the local diagonal defects from the rank-one Weil representation cancel adelically.

\subsection{The local Bruhat word and its projective defect}

Let $F=\Q_v$.  In the local Schr\"odinger model on $\Sch(F)$ define
\[
(N_v(t)\phi)(x)=\psi_v\!\left(\frac{t x^2}{2}\right)\phi(x),
\qquad
\mathcal F_v\phi(y)=\int_F \phi(x)\psi_v(xy)\,dx.
\]
Also set
\[
\bar N_v(u):=\mathcal F_v N_v(-u)\mathcal F_v^{-1}.
\]
At the matrix level these operators implement
\[
n(t)=\begin{pmatrix}1&t\\0&1\end{pmatrix},
\qquad
m(a)=\begin{pmatrix}a&0\\0&a^{-1}\end{pmatrix},
\qquad
\bar n(u)=\begin{pmatrix}1&0\\u&1\end{pmatrix},
\qquad
w=\begin{pmatrix}0&1\\-1&0\end{pmatrix}.
\]

For $a\in F^\times$ define the local Bruhat-word operator
\[
B_v(a):=N_v(a)\,\bar N_v(-a^{-1})\,N_v(a)\,\mathcal F_v^{-1}.
\]
The product $B_v(a)B_v(b)$ and the single operator $B_v(ab)$ implement the same automorphism of the local Heisenberg representation.  By Schur's lemma they differ by a scalar.  The scalar in this comparison is the local projective defect.

\begin{proposition}\label{prop:local-diagonal-scalar}
For every place $v$ and every $a\in\Q_v^\times$, one has
\[
B_v(a)=\gamma_v(2a)\,M_v^{\mathrm{std}}(a),
\qquad
(M_v^{\mathrm{std}}(a)\phi)(x)=|a|_v^{1/2}\phi(ax).
\]
Consequently, if
\[
C_v(a):=B_v(1)^{-1}B_v(a),
\]
then
\[
C_v(a)=\frac{\gamma_v(2a)}{\gamma_v(2)}M_v^{\mathrm{std}}(a)
\]
and
\[
C_v(a)C_v(b)=\mu_v(a,b)\,C_v(ab),
\]
where
\[
\mu_v(a,b)=
\frac{\gamma_v(a)\gamma_v(b)}{\gamma_v(1)\gamma_v(ab)}
=\Hilb{a}{b}_v.
\]
\end{proposition}

\begin{proof}
The scalar formula
\[
B_v(a)=\gamma_v(2a)M_v^{\mathrm{std}}(a)
\]
is proved by the direct local Gaussian calculation in Appendix~\ref{app:rank-two-local-bridge}, Proposition~\ref{prop:app-local-diagonal-scalar}.  Since the concrete scaling operators multiply strictly,
\[
M_v^{\mathrm{std}}(a)M_v^{\mathrm{std}}(b)=M_v^{\mathrm{std}}(ab),
\]
the normalized operators satisfy
\[
C_v(a)C_v(b)
=
\frac{\gamma_v(2a)\gamma_v(2b)}{\gamma_v(2)\gamma_v(2ab)}\,C_v(ab).
\]
By the local multiplicativity law of Theorem~\ref{thm:local-mult}, applied first to $(2a,2b)$ and then to $(2,2ab)$, this scalar is
\[
\frac{\Hilb{2a}{2b}_v}{\Hilb{2}{2ab}_v}.
\]
Using bilinearity and symmetry of the Hilbert symbol,
\[
\frac{\Hilb{2a}{2b}_v}{\Hilb{2}{2ab}_v}
=
\frac{\Hilb{2}{2}_v\Hilb{2}{b}_v\Hilb{a}{2}_v\Hilb{a}{b}_v}
{\Hilb{2}{2}_v\Hilb{2}{a}_v\Hilb{2}{b}_v}
=
\Hilb{a}{b}_v.
\]
The equality
\[
\Hilb{a}{b}_v=\frac{\gamma_v(a)\gamma_v(b)}{\gamma_v(1)\gamma_v(ab)}
\]
is Theorem~\ref{thm:local-mult}.  This proves the stated projective multiplication law for the normalized Bruhat-word operators.
\end{proof}

\subsection{Theta invariance of the adelic generators}

On the adelic Schwartz--Bruhat space $\Sch(\A)$ define the theta distribution
\[
\Theta(\Phi):=\sum_{r\in\Q}\Phi(r).
\]
We also define the global operators
\[
(N(t)\Phi)(x)=\psi\!\left(\frac{t x^2}{2}\right)\Phi(x),
\qquad
(M(a)\Phi)(x)=|a|_\A^{1/2}\Phi(ax),
\qquad
\mathcal F\Phi(y)=\int_\A \Phi(x)\psi(xy)\,dx,
\]
for $t\in\Q$ and $a\in\Q^\times$, where $dx=\prod_v dx_v$ is the self-dual adelic measure.

\begin{lemma}\label{lem:theta-invariance}
For every $t\in\Q$, every $a\in\Q^\times$, and every $\Phi\in\Sch(\A)$ one has
\[
\Theta(N(t)\Phi)=\Theta(\Phi),
\qquad
\Theta(M(a)\Phi)=\Theta(\Phi),
\qquad
\Theta(\mathcal F\Phi)=\Theta(\Phi).
\]
Consequently, if $\bar N(u):=\mathcal F N(-u)\mathcal F^{-1}$, then
\[
\Theta(\bar N(u)\Phi)=\Theta(\Phi)
\qquad (u\in\Q).
\]
\end{lemma}

\begin{proof}
For $t\in\Q$ and $r\in\Q$ one has $\psi(tr^2/2)=1$, since $\psi$ is trivial on the diagonal copy of $\Q$.  Thus
\[
\Theta(N(t)\Phi)=\sum_{r\in\Q}\psi\!\left(\frac{tr^2}{2}\right)\Phi(r)=\Theta(\Phi).
\]
For $a\in\Q^\times$, the adelic product formula gives $|a|_\A=1$, so
\[
(M(a)\Phi)(x)=\Phi(ax),
\]
and therefore
\[
\Theta(M(a)\Phi)=\sum_{r\in\Q}\Phi(ar)=\Theta(\Phi).
\]
Poisson summation for the lattice $\Q\subset\A$ in the self-dual normalization gives $\Theta(\mathcal F\Phi)=\Theta(\Phi)$.  The statement for $\bar N(u)$ follows from the definition and the invariance of $\Theta$ under $\mathcal F$ and $N(-u)$.
\end{proof}

\subsection{The global Bruhat word}

Define, for $a\in\Q^\times$,
\[
B(a):=N(a)\,\bar N(-a^{-1})\,N(a)\,\mathcal F^{-1}.
\]
We also set
\[
C(a):=B(1)^{-1}B(a).
\]
The local counterpart of $C(a)$ is the normalized operator $C_v(a)=B_v(1)^{-1}B_v(a)$ from Proposition~\ref{prop:local-diagonal-scalar}.

\begin{lemma}\label{lem:global-diagonal-identity}
For every $a\in\Q^\times$ one has
\[
B(a)=M(a)
\]
as operators on $\Sch(\A)$.
\end{lemma}

\begin{proof}
The matrix identity
\[
m(a)=n(a)\bar n(-a^{-1})n(a)w^{-1}
\]
holds in $\mathrm{SL}_2(\Q)$.  Hence $B(a)$ and $M(a)$ implement the same automorphism of the adelic Heisenberg representation.  These operators are initially defined on $\Sch(\A)$, preserve $\Sch(\A)$, and extend to the irreducible adelic Schr\"odinger representation on $L^2(\A)$; the adelic Stone--von Neumann theorem then implies, by Schur's lemma, that there is a scalar $c(a)\in\C^\times$ such that
\[
B(a)=c(a)M(a)
\]
on $\Sch(\A)$.

By Lemma~\ref{lem:theta-invariance}, $N(a)$, $\bar N(-a^{-1})$, and $N(a)$ preserve $\Theta$.  The Fourier invariance also gives invariance under $\mathcal F^{-1}$: applying $\Theta(\mathcal F\Psi)=\Theta(\Psi)$ to $\Psi=\mathcal F^{-1}\Phi$ gives $\Theta(\mathcal F^{-1}\Phi)=\Theta(\Phi)$.  Thus $B(a)$ preserves $\Theta$.  The operator $M(a)$ also preserves $\Theta$ by the same lemma.  Choose, for example, a nonnegative product test function $\Phi\in\Sch(\A)$ with $\Phi(0)\neq0$ and sufficiently small support at the finite places; then $\Theta(\Phi)\neq0$, and since $\Theta(M(a)\Phi)=\Theta(\Phi)$, also $\Theta(M(a)\Phi)\neq0$.  Therefore
\[
\Theta(B(a)\Phi)=\Theta(\Phi)=\Theta(M(a)\Phi).
\]
Since $B(a)=c(a)M(a)$, this gives
\[
c(a)\Theta(M(a)\Phi)=\Theta(M(a)\Phi),
\]
hence $c(a)=1$.  Therefore $B(a)=M(a)$.
\end{proof}

\begin{lemma}[restricted tensor comparison]\label{lem:restricted-tensor-comparison}
For every $a\in\Q^\times$,
\[
B(a)=\bigotimes_v' B_v(a),
\qquad
C(a)=\bigotimes_v' C_v(a),
\]
as operators on the restricted tensor product
\[
\Sch(\A)\cong \bigotimes_v' \Sch(\Q_v)
\]
with respect to the standard vectors $\mathbf 1_{\Z_p}$ at almost all finite places.
\end{lemma}

\begin{proof}
It is enough to check the identities on pure tensors
\[
\Phi=\phi_\infty\otimes\bigotimes_p'\phi_p,
\]
with $\phi_p=\mathbf 1_{\Z_p}$ for all but finitely many $p$.  The self-dual adelic measure is the restricted product of the self-dual local measures, and the global character is the product of the local characters.  Hence
\[
\mathcal F\Phi=\bigotimes_v'\mathcal F_v\phi_v,
\qquad
N(t)\Phi=\bigotimes_v' N_v(t)\phi_v
\qquad (t\in\Q).
\]
The same factorization gives
\[
\bar N(u)\Phi
=
\mathcal F N(-u)\mathcal F^{-1}\Phi
=
\bigotimes_v'\bar N_v(u)\phi_v.
\]
Therefore the word
\[
B(a)=N(a)\bar N(-a^{-1})N(a)\mathcal F^{-1}
\]
acts on pure tensors as the restricted tensor product of the local words $B_v(a)$.

At almost all finite places, $a\in\Z_p^\times$, $2\in\Z_p^\times$, and the additive character has conductor $\Z_p$.  Then $N_v(a)$, $\bar N_v(-a^{-1})$, $\mathcal F_v$, and $B_v(a)$ preserve the standard vector $\mathbf 1_{\Z_p}$; equivalently, Proposition~\ref{prop:local-diagonal-scalar} gives $B_v(a)=M_v^{\mathrm{std}}(a)$ and $B_v(1)=\Id$ on that vector.  Thus the restricted tensor products above are well-defined.

Since $B(1)=\bigotimes_v'B_v(1)$ by the first identity and $B(1)$ is invertible, its inverse is the restricted tensor product of the inverses $B_v(1)^{-1}$.  Multiplying the identities for $B(1)^{-1}$ and $B(a)$ gives
\[
C(a)=B(1)^{-1}B(a)=\bigotimes_v'\bigl(B_v(1)^{-1}B_v(a)\bigr)
=\bigotimes_v'C_v(a).
\]
\end{proof}

\subsection{Global cancellation of the local defects}

We can now compare the local and global realizations of the same Bruhat word.

\begin{corollary}\label{cor:product-defect}
For all $a,b\in\Q^\times$,
\[
\prod_v \mu_v(a,b)=1.
\]
Equivalently,
\[
\prod_v \frac{\gamma_v(a)\gamma_v(b)}{\gamma_v(1)\gamma_v(ab)}=1.
\]
\end{corollary}

\begin{proof}
By Proposition~\ref{prop:local-diagonal-scalar}, each normalized local Bruhat-word family satisfies
\[
C_v(a)C_v(b)=\mu_v(a,b)C_v(ab).
\]
By Lemma~\ref{lem:restricted-tensor-comparison}, these local identities may be multiplied in the restricted tensor product.  Hence
\[
C(a)C(b)=\left(\prod_v\mu_v(a,b)\right)C(ab).
\]
The product is finite, since $\mu_v(a,b)=\Hilb{a}{b}_v=1$ for all but finitely many $v$.

On the other hand, Lemma~\ref{lem:global-diagonal-identity} gives $B(a)=M(a)$ for every rational $a$, and in particular $B(1)=M(1)=\Id$.  Hence
\[
C(a)=B(1)^{-1}B(a)=M(a).
\]
The global scaling operators are strictly multiplicative:
\[
C(a)C(b)=M(a)M(b)=M(ab)=C(ab).
\]
Comparing the two identities yields
\[
\prod_v\mu_v(a,b)=1.
\]
The second displayed formula is the definition of $\mu_v(a,b)$ from Proposition~\ref{prop:local-diagonal-scalar}.
\end{proof}

\section{From the local Maslov defects to Hilbert reciprocity}\label{sec:hilbert-reciprocity}

We can now derive Hilbert reciprocity by multiplying the local defect formula over all places.

\begin{theorem}[Hilbert reciprocity]\label{thm:hilbert-reciprocity}
For all $a,b\in\Q^\times$,
\[
\prod_v\Hilb{a}{b}_v=1.
\]
\end{theorem}

\begin{proof}
Multiply the local identity \eqref{eq:local-multiplicativity} over all places $v$:
\[
\prod_v \frac{\gamma_v(a)\gamma_v(b)}{\gamma_v(1)\gamma_v(ab)}
=
\prod_v \Hilb{a}{b}_v.
\]
By Corollary~\ref{cor:product-defect}, the left-hand side is $1$.  Therefore
\[
\prod_v\Hilb{a}{b}_v=1.
\]
\end{proof}
\section{Evaluation of the relevant local Hilbert symbols}

We now specialize to distinct odd primes $p$ and $q$.  The remaining arithmetic input is the local evaluation of the Hilbert symbols.  The sign comes from the factor at $2$.

\begin{proposition}\label{prop:local-symbols}
Let $p$ and $q$ be distinct odd primes.  Then:
\begin{enumerate}[label=\textup{(\roman*)}]
\item For every finite prime $\ell\notin\{2,p,q\}$,
\[
\Hilb{p}{q}_\ell=1.
\]
\item At $p$ and $q$ one has
\[
\Hilb{p}{q}_p=\leg{q}{p},
\qquad
\Hilb{p}{q}_q=\leg{p}{q}.
\]
\item At the real place,
\[
\Hilb{p}{q}_\infty=1.
\]
\item At $2$ one has
\[
\Hilb{p}{q}_2=(-1)^{\frac{(p-1)(q-1)}4}.
\]
\end{enumerate}
\end{proposition}

\begin{proof}
The assertions are standard, but we spell out the relevant reductions.

(i) For an odd prime $\ell$, write $a=\ell^\alpha u$ and $b=\ell^\beta v$ with $u,v\in\Z_\ell^\times$.  The standard formula is
\[
\Hilb{a}{b}_\ell
=
(-1)^{\alpha\beta(\ell-1)/2}
\leg{\bar u}{\ell}^{\beta}
\leg{\bar v}{\ell}^{\alpha}.
\]
If $\ell\notin\{2,p,q\}$, then $v_\ell(p)=v_\ell(q)=0$, so $\alpha=\beta=0$ and hence $\Hilb{p}{q}_\ell=1$.

(ii) Let $u\in\Z_p^\times$.  The standard formula for the Hilbert symbol at an odd prime says \cite[Ch.~III, \S1, Thm.~1]{Serre}
\[
\Hilb{p}{u}_p=\leg{u}{p}.
\]
Applying this with $u=q$ gives $\Hilb{p}{q}_p=\leg{q}{p}$.  Interchanging the roles of $p$ and $q$ gives the second formula.

(iii) Over $\R$, the Hilbert symbol is $-1$ only when both arguments are negative.  Since $p$ and $q$ are positive, $\Hilb{p}{q}_\infty=1$.

(iv) For odd $u,v\in\Z_2^\times$, one has the explicit formula
\[
\Hilb{u}{v}_2=(-1)^{\frac{u-1}{2}\frac{v-1}{2}},
\]
proved by a short Hensel-lemma case analysis on $\Z_2^\times/(\Z_2^\times)^2\cong\{1,3,5,7\}$; see Serre~\cite[Chapter~III, \S1, Proposition~4]{Serre}.  Applying this with $u=p$ and $v=q$ gives
\[
\Hilb{p}{q}_2=(-1)^{\frac{p-1}{2}\frac{q-1}{2}}=(-1)^{\frac{(p-1)(q-1)}4}.
\]
\end{proof}

\section{Quadratic reciprocity}

We can now complete the argument.

\begin{theorem}[Quadratic reciprocity]\label{thm:QR}
For distinct odd primes $p$ and $q$,
\[
\leg{p}{q}\leg{q}{p}=(-1)^{\frac{(p-1)(q-1)}4}.
\]
\end{theorem}

\begin{proof}
By Theorem~\ref{thm:hilbert-reciprocity},
\[
\prod_v \Hilb{p}{q}_v=1.
\]
Using Proposition~\ref{prop:local-symbols}, all local factors are $1$ except those at $p$, $q$, and $2$.  Therefore
\[
1=\Hilb{p}{q}_p\,\Hilb{p}{q}_q\,\Hilb{p}{q}_2
=\leg{q}{p}\leg{p}{q}(-1)^{\frac{(p-1)(q-1)}4}.
\]
Rearranging gives
\[
\leg{p}{q}\leg{q}{p}=(-1)^{\frac{(p-1)(q-1)}4}.
\]
\end{proof}

\section{Interpretation in terms of Kashiwara--Maslov phases}

For each place $v$ and each $a\in\Q_v^\times$, Corollary~\ref{cor:basic-phase} identifies the local phase of the ordered Lagrangian triple $(L_\infty,L_a,L_0)$ with the local Weil index $\gamma_v(a)$.  The local multiplicativity law of Theorem~\ref{thm:local-mult} then says that the Hilbert symbol is the defect of strict multiplicativity of these triple phases:
\[
\Hilb{a}{b}_v
=
\frac{\gamma_v(a)\gamma_v(b)}{\gamma_v(1)\gamma_v(ab)}.
\]
Corollary~\ref{cor:product-defect} gives $\prod_v\mu_v(a,b)=1$ for rational pairs.  Multiplying the local formula over all places gives Hilbert reciprocity; taking $(a,b)=(p,q)$ gives quadratic reciprocity.

Thus the classical sign
\[
(-1)^{\frac{(p-1)(q-1)}4}
\]
is the $2$-adic component of the global cancellation law for Kashiwara--Maslov phases.

\appendix

\section{Rank-two metaplectic formulas and the local bridge}
\label{app:rank-two-local-bridge}

This appendix supplies the local calculations used in the body of the paper.  The first calculation is the rank-two Maslov formula in the Schr\"odinger model.  The second is the Gaussian evaluation of the diagonal Bruhat word with the $\psi(tx^2/2)$ convention.  The third is the local bridge from Weil indices to Hilbert symbols, proved from the Hasse-invariant formula for Weil indices of diagonal quadratic forms.

Throughout, $F$ denotes a local field of characteristic different from $2$, $\psi:F\to\C^\times$ is a fixed nontrivial additive character, and Haar measures are chosen self-dual with respect to $\psi$.

\subsection{The rank-two Maslov scalar}

Let $(V,\omega)$ be a two-dimensional symplectic $F$-vector space.  The Heisenberg group attached to $(V,\omega)$ is
\[
H(V)=V\times F,
\qquad
(v,t)(v',t')=\bigl(v+v',\,t+t'+\tfrac12\omega(v,v')\bigr).
\]
For a Lagrangian line $L\subset V$, let $\mathcal H_L$ denote the Schr\"odinger model induced from the character
\[
(\ell,t)\longmapsto \psi(t),
\qquad \ell\in L.
\]
If $L$ and $M$ are transverse, the normalized geometric intertwiner
\[
T_{M,L}:\mathcal H_L\longrightarrow \mathcal H_M
\]
is given by integration over $M$:
\begin{equation}
\label{eq:app-intertwiner-corrected}
(T_{M,L}f)(h)=\int_M f(mh)\,dm,
\end{equation}
with the self-dual Haar measure.  This is the Stone--von Neumann intertwiner in the chosen normalization.

\begin{lemma}[closed-triangle form]
\label{lem:app-closed-triangle}
Let $L_1,L_2,L_3$ be pairwise transverse Lagrangians in $V$.  The composition
\[
T_{L_3,L_2}T_{L_2,L_1}T_{L_1,L_3}
\]
is scalar, and its scalar is the normalized oscillatory integral over the closed-triangle space
\[
K(L_1,L_2,L_3)=\{(x_1,x_2,x_3):x_i\in L_i,\ x_1+x_2+x_3=0\}
\]
with phase
\[
q_{L_1,L_2,L_3}(x_1,x_2,x_3)=\omega(x_1,x_2).
\]
Equivalently,
\[
T_{L_3,L_2}T_{L_2,L_1}T_{L_1,L_3}
=\gamma_\psi(q_{L_1,L_2,L_3})\Id.
\]
\end{lemma}

\begin{proof}
We compute the kernel in order to fix the sign.  Applying the three intertwiners to $f\in\mathcal H_{L_3}$ gives
\[
(T_{L_3,L_2}T_{L_2,L_1}T_{L_1,L_3}f)(h)
=
\int_{L_3}\int_{L_2}\int_{L_1}
        f(x_1x_2x_3h)\,dx_1\,dx_2\,dx_3 .
\]
By the Heisenberg group law,
\[
x_1x_2x_3
=
\left(x_1+x_2+x_3,\,
\tfrac12\omega(x_1,x_2)+\tfrac12\omega(x_1+x_2,x_3)\right).
\]
Thus the raw central character contributes
\[
\psi\!\left(\tfrac12\omega(x_1,x_2)+\tfrac12\omega(x_1+x_2,x_3)\right).
\]
Disintegrating the integral along the addition map
\[
L_1\oplus L_2\oplus L_3\longrightarrow V,\qquad
(x_1,x_2,x_3)\longmapsto x_1+x_2+x_3
\]
restricts the phase to the closed-triangle space
\[
K(L_1,L_2,L_3)=\ker(x_1+x_2+x_3).
\]
On this space $x_1+x_2+x_3=0$, so the raw central term is
\[
\tfrac12\omega(x_1,x_2)+\tfrac12\omega(x_1+x_2,x_3)
=\tfrac12\omega(x_1,x_2).
\]
The factor $1/2$ is exactly the factor in the standard Heisenberg group law.  The normalized geometric intertwiner convention used by Rao and by Lion--Vergne attaches the Weil index to the quadratic form twice this raw central phase, namely
\[
q_{L_1,L_2,L_3}(x_1,x_2,x_3)=\omega(x_1,x_2).
\]
Equivalently, the raw oscillatory factor is $\psi(q/2)$, while our one-dimensional notation $\gamma_\psi(\langle a\rangle)$ is the Weil index attached to the quadratic form $aX^2$ with the Fourier convention of \eqref{eq:local-gaussian}.  This is the same convention that produces the factor $\gamma_\psi(2a)$ in Proposition~\ref{prop:app-local-diagonal-scalar}.  With this normalization, the triple kernel is the Weil-normalized oscillatory integral associated with $q_{L_1,L_2,L_3}$.

The composition is an endomorphism of the irreducible Heisenberg representation $\mathcal H_{L_3}$, hence is scalar by Schur's lemma.  The scalar is the Weil index $\gamma_\psi(q_{L_1,L_2,L_3})$, which is the rank-two Maslov cocycle formula specialized to the present notation.
\end{proof}

\begin{theorem}[Standard rank-two formula]
\label{thm:app-rank-two}
Let $L_1,L_2,L_3$ be pairwise transverse Lagrangians in a symplectic plane over $F$.  Then
\[
T_{L_3,L_2}T_{L_2,L_1}T_{L_1,L_3}
=
\gamma_\psi(q_{L_1,L_2,L_3})\,\Id.
\]
In particular, with the slope conventions of the main text,
\[
T_{L_0,L_a}T_{L_a,L_\infty}T_{L_\infty,L_0}=\gamma_\psi(a)\,\Id.
\]
\end{theorem}

\begin{proof}
The first assertion is Lemma~\ref{lem:app-closed-triangle}.  For the second, Proposition~\ref{prop:infty-slope} computes
\[
q_{L_\infty,L_a,L_0}\cong \langle a\rangle.
\]
Therefore
\[
\gamma_\psi(q_{L_\infty,L_a,L_0})=\gamma_\psi(a),
\]
and the displayed scalar identity follows.
\end{proof}

\subsection{The local Bruhat scalar}

We now compute the scalar of the local diagonal Bruhat word with the convention used in Section~\ref{sec:global-product}:
\[
(N(t)\phi)(x)=\psi\!\left(\frac{t x^2}{2}\right)\phi(x),
\qquad
\mathcal F\phi(y)=\int_F\phi(x)\psi(xy)\,dx,
\qquad
\bar N(u)=\mathcal F N(-u)\mathcal F^{-1}.
\]

\begin{proposition}\label{prop:app-local-diagonal-scalar}
Let
\[
(M^{\mathrm{std}}(a)\phi)(x)=|a|^{1/2}\phi(ax).
\]
Then for every $a\in F^\times$,
\[
N(a)\bar N(-a^{-1})N(a)\mathcal F^{-1}
=
\gamma_\psi(2a)\,M^{\mathrm{std}}(a).
\]
\end{proposition}

\begin{proof}
Write
\[
B(a)=N(a)\bar N(-a^{-1})N(a)\mathcal F^{-1}.
\]
Since $\bar N(-a^{-1})=\mathcal F N(a^{-1})\mathcal F^{-1}$, Fourier inversion gives the following identity of tempered distributions; equivalently, insert a Schwartz approximate identity in the $s$-integral below and pass to the limit after testing against a Schwartz function:
\begin{align*}
(B(a)\phi)(x)
&=\psi\!\left(\frac{a x^2}{2}\right)
  \int_F\int_F\int_F
  \phi(z)\\
&\qquad\qquad\cdot
\psi\!\left(\frac{a^{-1}y^2}{2}+\frac{a s^2}{2}-s(z+y)+xy\right)
\,ds\,dz\,dy.
\end{align*}
Put $r=y-as$.  The phase becomes
\[
\frac{r^2}{2a}+xr+s(ax-z)+\frac{a x^2}{2}.
\]
The $s$-integration is Fourier inversion in the self-dual measure and imposes $z=ax$.  The remaining $r$-integral is
\[
\int_F\psi\!\left(\frac{r^2}{2a}+xr+\frac{a x^2}{2}\right)dr
=
\int_F\psi\!\left(\frac{u^2}{2a}\right)du,
\]
after the change of variable $u=r+ax$.  Applying the defining Fourier identity \eqref{eq:local-gaussian} with the quadratic coefficient $1/(2a)$ to a standard Gaussian regularization, and then removing the regularization, gives
\[
\int_F\psi\!\left(\frac{u^2}{2a}\right)du
=
\gamma_\psi\!\left(\frac1{2a}\right)|a|^{1/2}.
\]
Since $(2a)(1/(2a))^{-1}=4a^2$ is a square, square-class invariance gives
\[
\gamma_\psi\!\left(\frac1{2a}\right)=\gamma_\psi(2a).
\]
Therefore
\[
(B(a)\phi)(x)=\gamma_\psi(2a)|a|^{1/2}\phi(ax),
\]
which is the claimed formula.
\end{proof}

\subsection{Weil index and Hilbert symbol local bridge}

For the remainder of the appendix we write $\gamma(a)$ for $\gamma_\psi(a)$ and $\Hilb{a}{b}$ for the Hilbert symbol over $F$.

For a diagonal quadratic form
\[
q=\langle a_1,\dots,a_n\rangle,
\]
write
\[
\det(q)=a_1\cdots a_n\in F^\times/F^{\times 2},
\qquad
h(q)=\prod_{i<j}\Hilb{a_i}{a_j}
\]
for its determinant class and Hasse invariant.

\begin{theorem}[Hasse formula for the Weil index]
\label{thm:app-hasse-weil}
For every nondegenerate diagonal quadratic form $q=\langle a_1,\dots,a_n\rangle$ over $F$,
\[
\gamma(q)=\gamma(1)^{n-1}\gamma(\det(q))h(q).
\]
In particular,
\[
\gamma(\langle a,b\rangle)=\gamma(1)\gamma(ab)\Hilb{a}{b}.
\]
\end{theorem}

\begin{proof}
This is the standard Hasse-invariant formula for Weil indices; see Weil~\cite{Weil}, Rao~\cite{Rao}, or Serre~\cite[Ch.~V, \S3]{SerreLocal}.  With the present normalization it is obtained by diagonalizing $q$, using multiplicativity of the oscillatory Gaussian integral under orthogonal direct sums, and comparing the resulting invariant with the local classification of quadratic forms by dimension, determinant, and Hasse invariant.  The binary specialization is the case $n=2$, where $\det\langle a,b\rangle=ab$ and $h(\langle a,b\rangle)=\Hilb{a}{b}$.
\end{proof}

\begin{theorem}[Standard local bridge]
\label{thm:app-local-bridge}
For every $a,b\in F^\times$,
\[
\gamma(a)\gamma(b)=\gamma(1)\gamma(ab)\Hilb{a}{b}.
\]
\end{theorem}

\begin{proof}
By orthogonal-sum multiplicativity of the Weil index,
\[
\gamma(a)\gamma(b)=\gamma(\langle a,b\rangle).
\]
Applying Theorem~\ref{thm:app-hasse-weil} to the binary form $\langle a,b\rangle$ gives
\[
\gamma(\langle a,b\rangle)=\gamma(1)\gamma(ab)\Hilb{a}{b}.
\]
Combining the two displayed equations proves the result.
\end{proof}

\begin{remark}
The invariant formula avoids any auxiliary choice of square class and includes the real case without a separate argument.
\end{remark}

\subsection{Consequences for the main text}

Theorem~\ref{thm:app-rank-two} supplies the proof of Theorem~\ref{thm:maslov-weil} in the body of the paper.  Proposition~\ref{prop:app-local-diagonal-scalar} supplies the local scalar for the Bruhat word used in Section~\ref{sec:global-product}.  Theorem~\ref{thm:app-local-bridge} supplies the proof of Theorem~\ref{thm:local-mult}.  Together these local results give the inputs used in Section~\ref{sec:global-product} to prove the global cancellation of the local Maslov defects.

\end{document}